\documentclass[10pt,a4paper,british]{article}
\usepackage{textcomp}
\usepackage[latin9]{inputenc}
\usepackage{color}
\usepackage{babel}
\usepackage{refstyle}
\usepackage{mathtools}
\usepackage{amsmath}
\usepackage{amsthm}
\usepackage{amssymb}
\usepackage{stmaryrd}
\usepackage[all]{xy}
\usepackage[bookmarks=true,bookmarksnumbered=true,bookmarksopen=true,bookmarksopenlevel=1,
 breaklinks=false,pdfborder={0 0 0},pdfborderstyle={},backref=false,colorlinks=true]
 {hyperref}
\hypersetup{pdftitle={Redundancy Rank},
 pdfauthor={Tal Cohen and Itamar Vigdorovich},
 pdfstartview={FitH}colorlinks=true,linkcolor=magenta,urlcolor=red,citecolor=blue}

\makeatletter

\AtBeginDocument{\providecommand\corref[1]{\ref{cor:#1}}}
\AtBeginDocument{\providecommand\lemref[1]{\ref{lem:#1}}}
\pdfpageheight\paperheight
\pdfpagewidth\paperwidth

\RS@ifundefined{subsecref}
  {\newref{subsec}{name = \RSsectxt}}
  {}
\RS@ifundefined{thmref}
  {\def\RSthmtxt{theorem~}\newref{thm}{name = \RSthmtxt}}
  {}
\RS@ifundefined{lemref}
  {\def\RSlemtxt{lemma~}\newref{lem}{name = \RSlemtxt}}
  {}

\numberwithin{equation}{section}
\numberwithin{figure}{section}
\theoremstyle{plain}
\newtheorem{thm}{\protect\theoremname}[section]
\theoremstyle{plain}
\newtheorem{conjecture}[thm]{\protect\conjecturename}
\theoremstyle{plain}
\newtheorem{cor}[thm]{\protect\corollaryname}
\theoremstyle{plain}
\newtheorem{prop}[thm]{\protect\propositionname}
\theoremstyle{plain}
\newtheorem{lem}[thm]{\protect\lemmaname}
\theoremstyle{definition}
\newtheorem{defn}[thm]{\protect\definitionname}
\theoremstyle{remark}
\newtheorem{rem}[thm]{\protect\remarkname}

\@ifundefined{date}{}{\date{}}
\usepackage{fullpage}
\date{}

\newref{lem}{refcmd={Lemma \ref{#1}}}
\newref{cor}{refcmd={Corollary \ref{#1}}}
\newref{thm}{refcmd={Theorem \ref{#1}}}

\AtBeginDocument{}
\def\moverlay{\mathpalette\mov@rlay}
\def\mov@rlay#1#2{\leavevmode\vtop{%
   \baselineskip\z@skip \lineskiplimit-\maxdimen
   \ialign{\hfil$\m@th#1##$\hfil\cr#2\crcr}}}
\newcommand{\charfusion}[3][\mathord]{
    #1{\ifx#1\mathop\vphantom{#2}\fi
        \mathpalette\mov@rlay{#2\cr#3}
      }
    \ifx#1\mathop\expandafter\displaylimits\fi}
\makeatother

\makeatother

\usepackage[style=alphabetic]{biblatex}
\providecommand{\conjecturename}{Conjecture}
\providecommand{\corollaryname}{Corollary}
\providecommand{\definitionname}{Definition}
\providecommand{\lemmaname}{Lemma}
\providecommand{\propositionname}{Proposition}
\providecommand{\remarkname}{Remark}
\providecommand{\theoremname}{Theorem}

\begin{document}
\global\long\def\res{\!\restriction}%

\global\long\def\actson{\curvearrowright}%

\global\long\def\Ad{\text{Ad}}%

\global\long\def\Prob#1{\text{Prob}\left(#1\right)}%

\global\long\def\Ch#1{\text{Ch}\left(#1\right)}%

\global\long\def\Tr#1{\text{Tr}\left(#1\right)}%

\global\long\def\Trleq#1{\text{Tr}_{\leq1}\left(#1\right)}%

\global\long\def\normalizer#1#2{\text{N}_{#1}\left(#2\right)}%

\global\long\def\FC#1{\text{FC}\left(#1\right)}%

\global\long\def\Sub#1{\text{Sub}\left(#1\right)}%

\global\long\def\conn#1{#1^{0}}%

\global\long\def\explain#1#2{\underset{\underset{\mathclap{#2}}{\downarrow}}{#1}}%

\global\long\def\Ind#1#2{\text{Ind}_{#2}^{#1}}%

\global\long\def\Res#1#2{\text{Res}_{#2}^{#1}}%

\global\long\def\N{\mathbb{N}}%
\global\long\def\Z{\mathbb{Z}}%
\global\long\def\Q{\mathbb{Q}}%
\global\long\def\R{\mathbb{R}}%
\global\long\def\C{\mathbb{C}}%
\global\long\def\H{\mathbb{H}}%
\global\long\def\T{\mathbb{T}}%
\global\long\def\P{\mathbb{P}}%

\global\long\def\bG{\mathbf{G}}%
\global\long\def\bH{\mathbf{H}}%
\global\long\def\bZ{\mathbf{Z}}%
\global\long\def\bK{\mathbf{K}}%

\global\long\def\BH{\mathcal{B}(\mathcal{H})}%

\global\long\def\sl#1{\mathrm{SL}(#1)}%

\global\long\def\slz{SL_{n}\left(\mathbb{Z}\right)}%
\global\long\def\slr{SL_{n}\left(\mathbb{R}\right)}%

\global\long\def\GC{\bG(\mathbb{C})}%
\global\long\def\G{\bG}%

\global\long\def\Gk{\mathbf{G}(k)}%

\global\long\def\fp{\mathfrak{p}}%

\global\long\def\zrr{m}%
\global\long\def\zr{\zrr^{\mathbb{C}}}%

\global\long\def\trr{m}%
\global\long\def\rr{m}%
\global\long\def\nrr{\mu}%

\global\long\def\rank{\mathrm{rank}}%
\global\long\def\ra{\mathrm{R}_{a}}%
\global\long\def\sr{\mathrm{R}_{s}}%

\title{On the Maximal Size of Irredundant Generating Sets in Lie Groups and
Algebraic Groups}
\author{Tal Cohen and Itamar Vigdorovich}
\maketitle
\begin{abstract}
We show the following dichotomy for a connected Lie group $G$: If
$G$ is amenable, then any topologically generating set $X\subset G$
of size larger than a fixed polynomial in the dimension of $G$ must
be redundant (i.e., a proper subset of $X$ still generates $G$).
If $G$ is non-amenable, then it admits arbitrarily large topologically
generating sets that are irredundant, and remain irredundant even
after applying Nielsen transformations. 

The polynomial bound for amenable groups is obtained by reduction
to finite simple groups of Lie type via strong approximation. This
partially answers two conjectures by Gelander on generation in compact
Lie groups and simple algebraic groups, and moreover shows that these
conjectures are implied by the Wiegold conjecture.

The construction of large Nielsen irredundant generating sets in non-amenable
groups is done by extending Minsky's work to higher rank Lie groups,
exhibiting dense representations in the domain of discontinuity of
the $\mathrm{Out}(F_{n})$-action on the character variety. 

\end{abstract}

\section{Introduction}

Let $G$ be a topological group. We say a subset $X\subset G$ is
\textit{generating }if the only closed subgroup of $G$ containing
$X$ is $G$. If $X$ admits a proper subset that is still generating,
we call it \textit{redundant}, and otherwise \textit{irredundant}.
How large irredundant generating sets in $G$ may be is measured by
the \textit{redundancy rank}
\[
\rr(G)=\sup\left\{ |X|:X\text{ is a finite irredundant generating set}\right\} .
\]

Estimating the value of $\rr(G)$ has been well-studied, but mostly
in the case $G$ is finite \cites{diaconis1998walks}{harper2023maximal}{lucchini2021bounding}{gill2023irredundant}{jambor2013minimal}{whiston2002maximal}.
For infinite $G$, the most basic problem is to determine whether
$\rr(G)$ is finite. For example, it is not hard to see that $\rr(\Z)=\infty$.
In general, however, the question of finiteness quickly becomes very
challenging, let alone obtaining effective bounds. 

Some deep questions arise when taking into account the natural action
of the automorphism group of the free group $\mathrm{Aut}(F_{n})$
on generating tuples $(x_{1},...,x_{n})\in G^{n}$. When identifying
$G^{n}\cong\mathrm{Hom}(F_{n},G)$, this action is simply given by
pre-composition. If $X\in G^{n}$ for some $n\in\mathbb{N}$, we
denote $\left|X\right|=n$ (so that $X\in G^{\left|X\right|}$). Observe
that, if $X$ is a generating tuple, then $\sigma(X)$ is generating
for every $\sigma\in\mathrm{Aut}(F_{\left|X\right|})$. A generating
tuple $X$ is said to be\textit{ Nielsen redundant} if there is some
$\sigma\in\mathrm{Aut}(F_{\left|X\right|})$ such that $\sigma(X)$
is redundant. Otherwise we say that $X$ is \textit{Nielsen irredundant}.
The \emph{Nielsen redundancy rank }is defined by 
\[
\nrr(G)=\sup\left\{ |X|:X\text{ is finite, generating, and Nielsen irredundant}\right\} .
\]
Clearly, \(\mathrm{d}(G)  \le\nrr(G)\le\rr(G)\),
where $\mathrm{d}(G)$ is the minimal size of a generating set. 

In this paper we are mainly interested in Lie groups, as well as algebraic
groups endowed with the Zariski topology. As we shall soon see, such
problems are in fact closely related to analogous problems in finite
groups. 

Our first result is the following.
\begin{thm}
\label{thm:master} There exist $a,b>0$ such that the following dichotomy
holds for any connected Lie group $G$.
\begin{enumerate}
\item \label{enu:amenable}If $G$ is amenable, then $\nrr(G)\le m(G)\le a\left(\dim G\right)^{b}$. 
\item \label{enu:non-amenable}If $G$ is not amenable, then $\nrr(G)=\rr(G)=\infty$. 
\end{enumerate}
\end{thm}

For a visual example, consider the rotation group $\mathrm{SO}(n)$.
It can be generated by two elements, but it is not hard to construct
strictly larger irredundant generating sets. For example, $\mathrm{SO}(3)$
can be generated irredundantly by $3$ elements, and $\mathrm{SO}(4)$
by $6$ elements. By Theorem \ref{thm:master} any generating set
of $\mathrm{SO}(n)$ of size at least a certain polynomial in $n$
must be redundant. Using further results on finite simple groups,
we show below that $\trr(\mathrm{SO}(3))=3$ and $\trr(\mathrm{SO}(4))=6$
(see Theorem \ref{thm:rotations}).

The redundancy rank for abelian $G$ is treated in \cite{CohVig}.
The fact that $\nrr(G)=\infty$ for $G=\mathrm{PSL}_{2}(\R)$ and
$G=\mathrm{PSL}_{2}(\C)$ follows from the work of Minsky \cite{minsky2013dynamics},
see also \cite[Section 9]{Lubotzky2011AutFnRepresentations} for a
beautiful overview. Other than that, Theorem \ref{thm:master} is
novel. \\

The invariant $\mu(G)$ is closely related to several important problems
in group theory \cite{gelander2024aut,Lubotzky2011AutFnRepresentations}.
It can be stated in terms of the connectedness of the product replacement
graph $X_{n}(G)$ of $n$-generating sets, and it is tightly connected
to questions on presentations of groups such as Gruenberg's questions.
The famous conjecture of Wiegold predicts that $\mu(G)=2$ for any
finite simple non-abelian group $G$; a stronger conjecture, due to
Pak \cite{pak2001we}, states that $\nrr(G)=d(G)$ for every finite
group. Gelander similarly conjectured that $\mu(G)=2$ for all connected
compact simple Lie groups, and that $\mu^{\C}(\bG)=2$ for any connected
simple complex algebraic group $\bG$, where $\mu^{\C}(\bG)$ (and
similarly $m^{\C}(\bG)$) denotes the analogous notion where $\bG$
is endowed with the Zariski topology. 

In what follows we will see that these conjectures are in fact more
than just analogous. Let $\bG$ be a connected simple complex algebraic
group. Recall $\mathbf{G}$ admits a structure of a split $\Z$-scheme.
Fixing such a structure, we denote by $\bG_{p}$ the $\mathbb{F}_{p}$-algebraic
group obtained by reduction modulo $p$. Recall also that $\bG$ admits
a unique real compact form $\bG_{c}$, so that $G=\bG_{c}(\R)$ is
a connected compact simple Lie group, and any such group arises in
this manner for a unique complex algebraic group. We show the following.
\begin{thm}
\label{thm:main-algebraic}Let $\bG$ be a simple connected complex
algebraic group, let $G=\bG_{c}(\R)$ be its compact real form, and
let $\tilde{\mathbf{G}}$ be its simply-connected cover. Then
\begin{align}
m(G) & \leq\zr(\bG)\leq\limsup_{p\;\mathrm{prime}}\rr(\tilde{\bG}_{p}(\mathbb{F}_{p})),\label{eq:m(G) bound}\\
\mu(G) & \leq\mu^{\C}(\bG)\leq\limsup_{p\;\mathrm{prime}}\mu(\tilde{\bG}_{p}(\mathbb{F}_{p})).\label{eq:mu(G) bound}
\end{align}
\end{thm}

We thus observe the following implications between the conjectures:
\[
\mathrm{W}\Rightarrow\bG\Rightarrow G,
\]
where $\mathrm{W}$ stands for the Wiegold conjecture that $\mu(G)=2$
for all non-abelian finite simple groups, $\mathbf{G}$ stands for
Gelander's conjecture that $\mu^{\C}(\bG)=2$ for every simple connected
complex algebraic group $\mathbf{G}$, and $G$ stands for Gelander's
conjecture that $\mu(G)=2$ for every connected compact simple Lie
group $G$ (see Conjecture \ref{conj:gelander-1} below).

Theorem \ref{thm:main-algebraic} allows us to use the existing literature
on the (Nielsen) redundancy rank of finite simple groups in order
to obtain novel results for Lie groups and algebraic groups. For example,
a recent theorem by Harper \cite{harper2023maximal} states that $\rr(\bG_{p}(\mathbb{F}_{p}))\leq a\left(\mathrm{rank}(\bG)\right)^{b}$
for every $\mathbf{G}$ and $p$ (with $a=10^{5}$ and $b=10$). We
thus obtain the bounds in Theorem \ref{thm:master}(\ref{enu:amenable}),
in the case $G$ is compact. 

With additional work, we obtain the following theorem on reductive
groups. 
\begin{thm}
\label{thm:AmenableAndAlgebraic}For every connected reductive complex
algebraic group $\mathbf{G}$, 
\[
\mu^{\C}(\bG)\leq\zr(\mathbf{G})\le a\cdot\rank(\bG)^{b}<\infty.
\]
\end{thm}

To put these results in better context, it is worth recalling Gelander's
conjectures as stated in \cite{gelander2024aut}:
\begin{conjecture}[{Gelander \cite[Questions 1.1, 2.9 and Conjecture 4.2]{gelander2024aut}}]
\label{conj:gelander-1} Let $n\geq3$, and let $F_{n}$ be the free
group on $n$ generators. Let $G$ be a connected compact simple Lie
group (respectively a connected simple $\mathbb{C}$-algebraic group).
Then, for any homomorphism $f:F_{n}\to G$ with dense (resp. Zariski
dense) image, there exists a non-trivial free product decomposition
$F_{n}=A*B$ such that $f(A)\subset G$ is dense (resp. Zariski dense).
\end{conjecture}

Thus, an immediate consequence of Theorems \ref{thm:AmenableAndAlgebraic}
and \ref{thm:master} is:
\begin{cor}
\label{cor:conj-gelander}Conjecture \ref{conj:gelander-1} holds
for any $n\in\N$ larger than a fixed polynomial in the complex rank
of $G$. 
\end{cor}

Corollary \ref{cor:conj-gelander} was independently obtained in \cite{CDMB}
with completely different methods: their proof mirrors the arguments
on finite groups via Jordan's theorem, whereas we reduce to the case
of finite simple groups of Lie type via strong approximation.\\

Using further results on finite simple groups, such as \cite{gilman1977finite,jambor2013minimal},
we are able to compute the (Nielsen) redundancy rank of several groups.
In particular, we show that Conjecture \ref{conj:gelander-1} holds
completely for $G=\mathrm{SO}(3)$ and $\mathbf{G}=\mathrm{SL}_{2}$
(for every $n\geqslant3$).
\begin{thm}
\label{thm:rotations}We have 
\end{thm}

\begin{itemize}
\item $\nrr^{\mathbb{C}}(\mathrm{SL}_{2})=\mu(\mathrm{SO}(3))=2$; equivalently,
Gelander's conjecture holds for these groups.
\item $\zr(\mathrm{SL}_{2})=\rr(\mathrm{SO}(3))=3$.
\item $\trr(\mathrm{U}(2))=4$.
\item $\trr(\mathrm{SO}(4))=6$.
\item $m(\mathrm{SU}(3))\leq6$.
\end{itemize}
An interesting question is whether one can obtain inequalities in
the reverse direction to Theorem \ref{thm:main-algebraic}. In other
words, whether the (Nielsen) redundancy rank of finite simple groups
can be bounded from above in terms of the (Nielsen) redundancy rank
of connected complex simple algebraic groups.\\

In order to establish Theorem \ref{thm:master}(\ref{enu:non-amenable}),
we consider primitive-stable representations, introduced by Minsky
for rank-one groups in \cite{minsky2013dynamics} and extended to
higher rank in \cite{kim2021primitive}. Along the way, we show the
following.
\begin{thm}
\label{Thm:domain of discontinuity}Let $G$ be a connected semisimple
linear Lie group without compact factors, and let $n\geq2$ even.
Consider the action of $\mathrm{Out}(F_{n})$ on the character variety
$\mathcal{X}(F_{n},G)=\mathrm{Hom}(F_{n},G)//G$. Then the domain
of discontinuity of this action contains an open set of dense primitive-stable
representations. 
\end{thm}

This extends the result of Minsky from rank 1 to higher rank, and
extends \cite[Theorem 1.1]{kim2021primitive}, which states that the
domain of discontinuity contains more than just Anosov representations.
Theorem \ref{Thm:domain of discontinuity} shows that the domain of
discontinuity goes beyond discrete faithful representations. 

In this paper, we have considered connected Lie groups. The case of
algebraic groups over $p$-adic fields (with the $p$-adic topology)
is also of considerable interest. It has been shown to exhibit behavior
quite different from the Archimedean case, a phenomenon sometimes
referred to as the \textquotedblleft Yair dichotomy\textquotedblright{}
\cites{glasner2009zeroone}{Lubotzky2011AutFnRepresentations}{gelander2024aut}.
It would therefore be interesting to understand to what extent the
results of this paper admit analogues in that setting.

\subsection*{Acknowledgments}

We would like to thank Nir Avni, Guy Kapon, Inkang Kim, Sungwoon Kim,
Alexander Lubotzky and Yair Minsky for valuable discussions. The first
author was co-funded by the European Union (ERC, Function Fields,
101161909). The second author was supported by NSF postdoctoral fellowship
grant DMS-2402368.

\section{Simple Algebraic Groups\label{sec:Simple-Groups}}

\subsection{Simple Simply-Connected Algebraic Groups\label{subsec:Simple-Simply-Connected-Algebrai}}

The purpose of this section is to prove Theorem \ref{thm:main-algebraic}
under the additional assumption that the algebraic group $\bG$ is
absolutely almost simple and simply-connected. We start by further
restricting to the case where the generating sets are required to
belong to an $S$-arithmetic subgroup.

\subsubsection{The Arithmetic Case}

Let $F$ be a finite Galois extension of $\Q$, and $\mathcal{O}_{F}$
its ring of integers. Fix a finite set $S$ of non-zero prime ideals
in $\mathcal{O}_{F}$ and set $A=\mathcal{O}_{F}S^{-1}$, the localisation
of $\mathcal{O}_{F}$ at $S$. Let $\mathcal{P}(A)$ denote the set
of non-zero prime ideals in $A$ and let $\mathcal{P}_{1}(A)\subset\mathcal{P}(A)$
consist of those primes with residue degree $1$, namely those $\mathfrak{p}\in\mathcal{P}(A)$
for which the field $A/\mathfrak{p}$ is prime. By the Chebotarev
density theorem the set $\mathcal{P}_{1}(A)$ is infinite. 

Let $\bG$ be a connected, simply-connected, absolutely almost simple
group scheme defined over $\Z$. Given $\mathfrak{p}\in\mathcal{P}(A)$,
let $\bG_{\mathfrak{p}}$ denote the algebraic group over the field
$A/\mathfrak{p}$ obtained via extension of scalars to $A$ and reduction
to $A/\mathfrak{p}$. If $\mathfrak{p}\in\mathcal{P}_{1}(A)$ we denote
$\bG_{p}=\bG_{\mathfrak{p}}$, where $p=|A/\mathfrak{p}|$.

In what follows, for a given $X\in\bG(A)^{n}$, we denote by $X_{\mathfrak{p}}\in\bG(A/\mathfrak{p})^{n}$
the corresponding tuple after applying the map $\bG(A)\to\bG_{\mathfrak{p}}(A/\mathfrak{p})$
entrywise. The following proposition is a variation of Lubotzky's
`one for almost all' trick \cite{lubotzky1999one}.
\begin{prop}
\label{prop:generation}There exists an exceptional finite set of
prime ideals $\mathcal{Q}\subset\mathcal{P}(A)$, depending only on
$\bG$ and $A$, such that the following statements are equivalent
for $X\in\bG(A)^{n}$.
\begin{enumerate}
\item $X$ Zariski generates $\bG$. 
\item $X_{\mathfrak{p}}$ generates $\bG_{\mathfrak{p}}(A/\mathfrak{p})$
for almost every $\mathfrak{p}\in\mathcal{P}_{1}(A)$. 
\item $X_{\mathfrak{p}}$ generates $\bG_{\mathfrak{p}}(A/\mathfrak{p})$
for some $\mathfrak{p}\in\mathcal{P}_{1}(A)\backslash\mathcal{Q}$.
\end{enumerate}
\end{prop}

\begin{proof}
For a prime integer $p$, let $K_{p}$ be the $p^{\mathrm{th}}$ principal
congruence subgroup, that is the kernel of the map $\bG(\Z_{p})\to\bG(\mathbb{F}_{p})$.
Then $K_{p}$ is the Frattini subgroup of $\bG(\Z_{p})$ for all primes
$p$ outside a finite exceptional set of primes $Q_{0}$ \cite{weigel1995finite}
(see also \cite{weigel1996profinite} and \cite{lubotzky1999one}).
We let $\mathcal{Q}\subset\mathcal{P}(A)$ be the set of primes $\mathfrak{p}$
with $\mathrm{char}(A/\mathfrak{p})=[\Z:\mathfrak{p}\cap\Z]\in Q_{0}$.
Clearly $\mathcal{Q}$ is finite. 

Now let $X\in\bG(A)^{n}$. The implication $1\Rightarrow2$ follows
from the strong approximation theorem \cite{weisfeiler1984strong}
(see \cite{avni2025mixed} for a new elegant proof of this classical
theorem). The implication $2\Rightarrow3$ is trivial. 

It is left to show $3\Rightarrow1$. Fix $\mathfrak{p}\in\mathcal{P}_{1}(A)\backslash\mathcal{Q}$
such that $X_{\mathfrak{p}}$ generates $\bG_{\mathfrak{p}}(A/\mathfrak{p})$.
Consider the completion of $A$ with respect to the valuation corresponding
to $\mathfrak{p}$, and denote it $\bar{A}^{\fp}$. We have the following
diagram
\[
\xymatrix{\bG(A)\ar[r]\ar[d] & \bG(\bar{A}^{\fp})\ar[d]\\
\bG(A/\fp)\ar[r]^{\tilde{f}} & \bG(\bar{A}^{\fp}/\fp)
}
\]
Note that the bottom arrow is an isomorphism. Since $\fp\in\mathcal{P}_{1}(A)$,
we have that $\bar{A}^{\mathfrak{p}}\cong\Z_{p}$ and $\bar{A}^{\fp}/\fp=\mathbb{F}_{p}$,
where $p=|A/\mathfrak{p}|$. Since $\mathfrak{p}\notin\mathcal{Q}$,
the kernel of the vertical map on the right-hand side of the diagram
is the Frattini subgroup of $\bG(\bar{A}^{\mathfrak{p}})$. Thus,
when viewing $X$ as a tuple of elements in $\bG(\bar{A}^{\fp})$,
it generates it in the profinite topology, so that $X$ certainly
generates $\bG$ in the Zariski topology. 
\end{proof}

\begin{prop}
\label{prop:redundancy}For $X\in\bG(A)^{n}$, the following are equivalent:
\begin{enumerate}
\item $X$ is a (Nielsen) irredundant Zariski generating set of $\bG$.
\item $X_{\mathfrak{p}}$ is a (Nielsen) irredundant generating set of $\bG_{\mathfrak{p}}(A/\mathfrak{p})$
for almost every $\mathfrak{p}\in\mathcal{P}_{1}(A)$. 
\item $X_{\mathfrak{p}}$ is a (Nielsen) irredundant generating set of $\bG_{\mathfrak{p}}(A/\mathfrak{p})$
for infinitely many $\mathfrak{p}\in\mathcal{P}_{1}(A)$. 
\end{enumerate}
\end{prop}

\begin{proof}
For $X\in\bG(A)^{n}$ and $\sigma\in\mathrm{Aut}(F_{n})$, we have
$\sigma(X)_{\mathfrak{p}}=\sigma(X_{\mathfrak{p}})$, so we may consider
redundancy and Nielsen redundancy simultaneously. Let $Q$ be the
finite exceptional set given in Proposition \ref{prop:generation}.

For $1\Rightarrow2$, let $X$ be a (Nielsen) irredundant Zariski
generating set of $\bG$. By Proposition \ref{prop:generation} ,
$X_{\mathfrak{p}}$ generates $\mathbf{G}_{\mathfrak{p}}(A/\mathfrak{p})$
for almost every $\mathfrak{p}\in\mathcal{P}_{1}(A)$. Moreover, for
every $\mathfrak{p}\in\mathcal{P}_{1}(A)\backslash\mathcal{Q}$ such
that $X_{\mathfrak{p}}$ generates $\mathbf{G}_{\mathfrak{p}}(A/\mathfrak{p})$,
$X_{\mathfrak{p}}$ must be (Nielsen) irredundant: otherwise we would
get by the defining property of $\mathcal{Q}$ that $X$ were (Nielsen)
redundant, contrary to our assumptions. 

The implication $2\Rightarrow3$ is trivial. For $3\Rightarrow1$,
assume that, for infinitely many $\mathfrak{p}\in\mathcal{P}_{1}(A)$,
$X_{\mathfrak{p}}$ is a (Nielsen) irredundant generating set of $\bG_{\mathfrak{p}}(A/\mathfrak{p})$.
Then, for some $\mathfrak{p}\in\mathcal{P}_{1}(A)\backslash\mathcal{Q}$,
$X_{\mathfrak{p}}$ generates $\bG_{\mathfrak{p}}(A/\mathfrak{p})$,
which means $X$ Zariski generates $\mathbf{G}$. Assume by contradiction
that $X$ is (Nielsen) redundant. Then by Proposition \ref{prop:generation},
$X_{\mathfrak{p}}$ is (Nielsen) redundant for almost every $\mathfrak{p}\in\mathcal{P}_{1}(A)$,
contrary to the assumption $X_{\mathfrak{p}}$ is (Nielsen) irredundant
for infinitely many $\mathfrak{p}\in\mathcal{P}_{1}(A)$. 
\end{proof}
We thus obtain:
\begin{cor}
\label{thm:liminf and limsup}Let $X\in\bG(A)^{n}$ be Zariski generating
for some $n\in\mathbb{N}$. If $X$ is irredundant, then
\[
|X|\leqslant\liminf_{\fp\in\mathcal{P}_{1}(A)}\rr(\bG_{\mathfrak{p}}(A/\mathfrak{p}))\leqslant\limsup_{p\in\N,\text{ prime}}\rr(\bG_{p}(\mathbb{F}_{p})).
\]
If $X$ is Nielsen irredundant, then 
\[
|X|\leqslant\liminf_{\fp\in\mathcal{P}_{1}(A)}\nrr(\bG_{\mathfrak{p}}(A/\mathfrak{p}))\leqslant\limsup_{p\in\N,\text{ prime}}\nrr(\bG_{p}(\mathbb{F}_{p})).
\]
\end{cor}

\begin{proof}
The inequality $|X|\leqslant\liminf_{\fp\in\mathcal{P}_{1}(A)}\rr(\bG_{\mathfrak{p}}(A/\mathfrak{p}))$
(respectively $|X|\leqslant\liminf_{\fp\in\mathcal{P}_{1}(A)}\nrr(\bG_{\mathfrak{p}}(A/\mathfrak{p}))$)
simply means that $|X|\leq\rr(\bG_{\mathfrak{p}}(A/\mathfrak{p}))$
(respectively $|X|\leq\nrr(\bG_{\mathfrak{p}}(A/\mathfrak{p}))$)
for almost all $\mathfrak{p}\in\mathcal{P}_{1}(A)$, which follows
from Proposition \ref{prop:redundancy}. The inequalities $\liminf_{\fp\in\mathcal{P}_{1}(A)}\rr(\bG_{\mathfrak{p}}(A/\mathfrak{p}))\leqslant\limsup_{p\in\N,\text{ prime}}\rr(\bG_{p}(\mathbb{F}_{p}))$
and $\liminf_{\fp\in\mathcal{P}_{1}(A)}\nrr(\bG_{\mathfrak{p}}(A/\mathfrak{p}))\leqslant\limsup_{p\in\N,\text{ prime}}\nrr(\bG_{p}(\mathbb{F}_{p}))$
are obtained by restricting to (the infinitely many) primes $p$ which
are totally split over $\mathcal{O}_{F}$. For such primes $p$ we
have $\bG_{p}(\mathbb{F}_{p})=\bG_{\mathfrak{p}}(A/\mathfrak{p})$
where $\mathfrak{p}\in\mathcal{P}_{1}(A)$ can be any of the prime
ideals appearing in the prime decomposition of $p$ in $\mathcal{O}_{F}$. 
\end{proof}

\subsubsection{The Non-Arithmetic Case\label{subsec:The-Non-Arithmetic-Case}}

We now turn to consider tuples $X$ that do not necessarily lie inside
an arithmetic subgroup. This is done via the notion of specialisation.
\begin{lem}
\label{lem:LL} Let $\bG$ be a connected, simply-connected, absolutely
almost simple, $\Z$-group scheme, and let $X\in\bG(\C)^{n}$ be a
(Nielsen) irredundant Zariski generating set for $\bG$ for some $n\in\mathbb{N}$.
Then there exists a number field $F$, a finite set of places $S$
of $F$, a ring $R\subseteq\mathbb{C}$ satisfying $X\subseteq\mathbf{G}(R)$,
and a map $\varphi:\mathbf{G}(R)\to\bG(\mathcal{O}_{F}S^{-1})$ such
that the image of $X$ under $\varphi$ (applied entrywise) is still
(Nielsen) irredundant Zariski generating.
\end{lem}

\begin{proof}
Fix a faithful representation $\mathbf{G}(\mathbb{C})\hookrightarrow\mathrm{GL}_{d}(\mathbb{C})$
defined over $\mathbb{Q}$. Let $\Gamma$ be the subgroup of $\bG(\C)$
generated abstractly by $X$, and let $R$ denote the ring generated
by the entries of the elements of $X$. Then $R$ is a finitely generated
subring of $\C$ and $\Gamma\subset\bG(R)$. By \cites[Theorem 4.1]{LarLub},
there exists a number field $F$ and a ring homomorphism $\phi:R\to F$
such that the Zariski closure of $\varphi(\Gamma)$ is isomorphic
over $\C$ to $\GC$, where we denote by $\varphi$ the homomorphism
obtained by applying $\phi$ to the entries of the matrix. (Observe
that we are using different notations: what they denote by $A$ we
denote by $R$, what they denote by $K$ we simply write as $\mathrm{frac}(R)$,
and what they denote by $k$ we denote by $F$. Moreover, observe
that they allow `simple' groups to have finite centre; see the \emph{Notations
and conventions} at the end of page 5.) Since the only $\C$-subgroup
of $\bG$ that is isomorphic to itself is $\bG$, we get that $\varphi(\Gamma)$
is Zariski dense in $\bG(\C)$$.$ In other words, $\varphi(X)$ (where
$\varphi$ is applied entrywise) Zariski generates $\bG$, and since
this set is finite and contained in $\bG(F)$, it is moreover contained
in $\bG(\mathcal{O}_{F}S^{-1})$ for some finite set of places $S$.

It is left to show $\varphi(X)$ is (Nielsen) irredundant. Since $\varphi$
commutes with the action of $\mathrm{Aut}(F_{\left|X\right|})$, it
is enough to consider regular redundancy. Thus, consider $X$ as a
subset of $\mathbf{G}(\mathbb{C})$ and assume by contradiction there
is a proper subset $Y\subsetneq X$ such that $\varphi(Y)$ generates
a Zariski dense subgroup. 

Since $X$ is irredundant, it follows $\left\langle Y\right\rangle $
is \emph{not }Zariski dense in $\mathbf{G}(\mathbb{C})$. Since $\bG(\mathrm{frac}(R))$
is Zariski dense in $\bG(\C)$, it follows that $\left\langle Y\right\rangle $
is not Zariski dense in $\bG(\mathrm{frac}(R))$. Thus, there exists
a polynomial $p\ne0$ in the entries of $\G$ with coefficients in
the fraction field $\mathrm{frac}(R)$ such that $p(x)=0$ for every
$x\in\left\langle Y\right\rangle $. By multiplying by an element
of $R$, we may assume that the coefficients of $p$ are in $R$.

By \cites[Theorem 5.1]{AtiMac}, it is possible to extend $\phi$
to a homomorphism $\psi:B\to\mathbb{C}$ for some valuation ring $B\subseteq\mathrm{frac}(R)$.
(Recall that $B$ is a valuation ring of $\mathrm{frac}(R)$ if, for
every $x\neq0$, we have $x\in B$ or $x^{-1}\in B$.) For every two
coefficients $a_{i},a_{j}$ in $p$, either $a_{i}/a_{j}$ or $a_{j}/a_{i}$
(or both) are in $B$. Therefore, one may show inductively that there
is $i_{0}$ such that $a_{i_{0}}$ divides in $B$ all the coefficients
of $p$. Consider $\tilde{p}=\frac{1}{a_{i_{0}}}p$. This is a polynomial
over $B$, so we can apply $\psi$ to its coefficients. Denote by
$\tilde{p}^{\psi}$ the polynomial $\tilde{p}$ after we apply $\psi$
to its coefficients. Observe $\tilde{p}^{\psi}$ is not the zero polynomial,
since the $i_{0}^{\mathrm{th}}$ coefficient of $\tilde{p}$ (and
hence also of $\tilde{p}^{\psi}$) is $1$. However, for every $x\in\left\langle Y\right\rangle $,
\[
\tilde{p}^{\psi}(\varphi(x))=\psi(\tilde{p}(x))=\psi(0)=0.
\]
We get that $\tilde{p}^{\psi}$ is zero on $\left\langle \varphi(Y)\right\rangle =\varphi(\left\langle Y\right\rangle )$.
Since the latter is Zariski dense, we get that $\tilde{p}^{\psi}$
is the zero polynomial. Contradiction!
\end{proof}
We therefore get
\begin{thm}
\label{thm:simply-connected-case}Let $\bG$ be a connected, simply-connected,
absolutely almost simple group scheme over $\Z$. Then
\[
\zr(\bG)\leq\limsup_{p\in\mathbb{N}\text{ prime}}\rr(\bG_{p}(\mathbb{F}_{p}))
\]
and
\[
\nrr^{\mathbb{C}}(\bG)\leq\limsup_{p\in\mathbb{N}\text{ prime}}\nrr(\bG_{p}(\mathbb{F}_{p})).
\]
\end{thm}

\begin{proof}
Let $X$ be a (Nielsen) irredundant Zariski generating set for $\bG(\C)$.
By Lemma \ref{lem:LL}, there exists a number field $F$, a finite
set of places $S$ of $F$, a ring $R\subseteq\mathbb{C}$ satisfying
$X\subseteq\mathbf{G}(R)$, and a homomorphism $\varphi:\mathbf{G}(R)\to\bG(\mathcal{O}_{F}S^{-1})$
such that $\varphi(X)$ is still Zariski generating and (Nielsen)
irredundant in $\bG$ (where we apply $\varphi$ entrywise). Up to
replacing $F$ with a larger field extension, we may assume that $F/\Q$
is Galois. We apply Corollary \ref{thm:liminf and limsup} and get
\[
|X|=|\varphi(X)|\leq\limsup_{p\in\mathbb{N}\text{ prime}}\rr(\bG_{p}(\mathbb{F}_{p}))
\]
in the case $X$ is irredundant, and 
\[
|X|=|\varphi(X)|\leq\limsup_{p\in\mathbb{N}\text{ prime}}\nrr(\bG_{p}(\mathbb{F}_{p}))
\]
in the case $X$ is Nielsen irredundant. Since $X$ was arbitrary,
we are done.
\end{proof}
Using the state-of-the-art estimates on the redundancy rank of finite
groups we deduce the following consequence.
\begin{cor}
\label{cor:simply-connected-ab-al-simple-case}Let $\bG$ be a connected,
simply-connected, absolutely almost simple group scheme over $\Z$.
Then
\[
\nrr^{\mathbb{C}}(\mathbf{G})\leq\zr(\bG)\leq10^{5}\mathrm{rank}(\bG)^{10}.
\]
\end{cor}

\begin{proof}
By \cite[Theorem 2]{harper2023maximal}, $\rr(\bG_{p}(\mathbb{F}_{p}))\leq10^{5}\rank(\bG)^{10}$
for every prime $p$. The result therefore follows from Theorem \ref{thm:simply-connected-case}.
\end{proof}

\subsection{Isogenies\label{subsec:Isogenies}}

\textcolor{magenta}{}From now on, by an algebraic group we will always
mean an affine algebraic group defined over a field of characteristic
zero.

Let $\bG$ be a connected semisimple algebraic group. Recall there
exists a connected, semisimple, simply-connected algebraic group $\tilde{\bG}$,
a connected adjoint algebraic group $\bar{\bG}$, and isogenies $\tilde{\bG}\to\bG\to\bar{\bG}$.
We will now show that $\zr(\mathbf{G})=\zr(\bar{\mathbf{G}})=\zr(\tilde{\mathbf{G}})$
and $\nrr^{\mathbb{C}}(\mathbf{G})=\nrr^{\mathbb{C}}(\bar{\mathbf{G}})=\nrr^{\mathbb{C}}(\tilde{\mathbf{G}})$.
In particular, 
\[
\zr(\bG)\leq10^{5}\mathrm{rank}(\bG)^{10}.
\]

\begin{lem}
Let $G=\mathbf{G}(\mathbb{C})$ be a perfect algebraic group, and
let $H\leqslant G$ be some abstract, Zariski dense subgroup. Then
$[H,H]$ is Zariski dense.
\end{lem}

\begin{proof}
Let $V\subseteq G$ be a nonempty Zariski open subset. Then it contains
an element of $G=[G,G]$, say $[x_{1}x_{2}]\cdots[x_{2n-1}x_{2n}]$.
Therefore, the map
\begin{align*}
\varphi:G^{2n} & \to G\\
\varphi(g_{1},\dots,g_{2n}) & =[g_{1}g_{2}]\cdots[g_{2n-1}g_{2n}]
\end{align*}
`reaches' $V$; i.e., $\varphi^{-1}(V)$ is nonempty. Since $H^{2n}$
is Zariski dense in $G^{2n}$, we get that $H^{2n}\cap\varphi^{-1}(V)$
is nonempty. In other words, there are $h_{1},\dots,h_{2n}$ such
that $[h_{1}h_{2}]\cdots[h_{2n-1}h_{2n}]\in V$, as needed.
\end{proof}
Since the kernel of an isogeny is central, we get:
\begin{cor}
\label{cor:isogeny}Let $G=\mathbf{G}(\mathbb{C})$ be a perfect algebraic
group, $f:G\to H$ an isogeny (onto some other algebraic group $H=\mathbf{H}(\mathbb{C})$).
If $g_{1},\dots,g_{n}\in G$ are such that $f(g_{1}),\dots,f(g_{n})$
generate a Zariski dense subgroup of $H$, then $g_{1},\dots,g_{n}$
generate a Zariski dense subgroup of $G$. In particular, $\zr(\mathbf{G})=\zr(\mathbf{H})$
and $\nrr^{\mathbb{C}}(\mathbf{G})=\nrr^{\mathbb{C}}(\mathbf{H})$.
\end{cor}

\subsection{Compact Groups\label{subsec:Compact-Groups}}

If $\bG$ is a $K$-algebraic group, we denote $\zrr^{K}(\bG)=\sup\left\{ |X|:X\subset\bG(K)\text{ is finite, Zariski generating and irredundant}\right\} $.
It is easy to see that $\zrr^{K}(\mathbf{G})\leqslant\zrr^{\mathbb{C}}(\mathbf{G})$
for every field $K$ of characteristic zero.
\begin{proof}[Proof of Theorem \ref{thm:main-algebraic}]
The inequalities on algebraic groups follow from Theorem \ref{thm:simply-connected-case}
and Corollary \ref{cor:isogeny}. It is well known that every compact
Lie group $G$ admits a unique structure of an algebraic group, $G=\mathbf{G}(\mathbb{R})$,
and that closed subgroups of $G$ are also Zariski closed (see, e.g.,
\cites(Ch. 5, \S 2.5\textdegree., Theorem 12 and Problem 24){OniVin}). This means
that subgroups of $G$ are dense if and only if they are Zariski dense.
Thus, for a connected compact simple Lie group $G$, we obtain
\begin{align*}
m(G) & =\zrr^{\mathbb{R}}(\mathbf{G})\leq\zr(\bG)=\zr(\tilde{\bG})\leq\limsup_{p\;\mathrm{prime}}\rr(\tilde{\bG}_{p}(\mathbb{F}_{p})),\\
\mu(G) & =\nrr^{\mathbb{R}}(\mathbf{G})\leq\mu^{\C}(\bG)=\mu^{\C}(\tilde{\bG})\leq\limsup_{p\;\mathrm{prime}}\mu(\tilde{\bG}_{p}(\mathbb{F}_{p})).\qedhere
\end{align*}
\end{proof}

\section{Reductive Algebraic Groups\label{sec:Reductive-and-Amenable}}

\subsection{Semisimple Algebraic Groups \label{subsec:Semisimple-Algebraic-Groups}}

\begin{lem}
Let $G_{1}=\mathbf{G}_{1}(\mathbb{C})$, $G_{2}=\mathbf{G}_{2}(\mathbb{C})$
be connected adjoint simple algebraic groups and let $R<G_{1}\times G_{2}$
be a proper algebraic subgroup that projects onto both $G_{1}$ and
$G_{2}$. Then there is an isomorphism $f:G_{1}\to G_{2}$ such that
$R=\left\{ (g,f(g))\middle|g\in G_{1}\right\} $.
\end{lem}

\begin{proof}
Denote by $\pi_{i}:G_{1}\times G_{2}\to G_{i}$ the projections. Since
$R$ is a proper subgroup and projects onto $G_{1}$ and $G_{2}$,
it is impossible to have either $G_{2}\subseteq R$ or $G_{1}\subseteq R$.
Thus $R\cap G_{i}$ is a proper normal subgroup of $G_{i}$, which
must therefore be trivial. This means that $\pi_{i}\res_{R}:R\to G_{i}$
is an isomorphism of algebraic groups. Setting $f=\pi_{2}\circ(\pi_{1}\res_{R})^{-1}$,
we get the desired result.
\end{proof}
\begin{lem}
Let $\mathbf{G}_{1},\mathbf{G}_{2}$ be connected simple algebraic
groups. Then $\zr(\mathbf{G}_{1}\times\mathbf{G}_{2})=\zr(\mathbf{G}_{1})+\zr(\mathbf{G}_{2})$.
\end{lem}

\begin{proof}
The inequality $\geqslant$ is straightforward: if $S_{1}$ and $S_{2}$
are two irredundant generating sets of $G_{1}=\mathbf{G}_{1}(\mathbb{C})$
and $G_{2}=\mathbf{G}_{2}(\mathbb{C})$ respectively, then $\left(S_{1}\times\left\{ 1\right\} \right)\cup\left(\left\{ 1\right\} \times S_{2}\right)$
is an irredundant generating set for $G_{1}\times G_{2}$.

The inequality $\leqslant$ requires some more care. Let $Z_{1},Z_{2}$
be the centres of $G_{1},G_{2}$ respectively, so that $Z_{1}\times Z_{2}$
is the centre of $G_{1}\times G_{2}$. By \corref{isogeny}, we have
$\zr(\mathbf{G}_{1})=\zr(\bar{\mathbf{G}}_{1})$, $\zr(\mathbf{G}_{2})=\zr(\bar{\mathbf{G}}_{2})$
and $\zr(\mathbf{G}_{1}\times\mathbf{G}_{2})=\zr(\bar{\mathbf{G}}_{1}\times\bar{\mathbf{G}}_{2})$.
In other words, we may assume $\mathbf{G}_{1}$ and $\mathbf{G}_{2}$
are adjoint.

Set $n_{i}=\zr(\mathbf{G}_{i})$, and suppose $S\coloneqq\left\{ g_{1},\dots,g_{n}\right\} \subseteq G\coloneqq G_{1}\times G_{2}$
generate a Zariski dense subgroup of $G$ with $n>n_{1}+n_{2}$.
Denote by $\pi_{i}:G\to G_{i}$ the projections. We will find a subset
$S'\subseteq S$ of size at most $n_{1}+n_{2}$ such that $\pi_{i}(\left\langle S'\right\rangle )$
is Zariski dense in $G_{i}$, and such that there is no isomorphism
$f:G_{1}\to G_{2}$ for which $f(x_{1})=x_{2}$ for every $(x_{1},x_{2})\in S'$.
We will then deduce the Zariski closure of $\left\langle S'\right\rangle $
must be all of $G_{1}\times G_{2}$ from the previous lemma.

By assumption, there is a subset $S_{1}\subseteq S\coloneqq\left\{ g_{1},\dots,g_{n}\right\} $
of size at most $n_{1}$ such that $\pi_{1}\left(\left\langle S_{1}\right\rangle \right)$
is Zariski dense in $G_{1}$. Similarly, there is a subset $S_{2}\subseteq S$
of size at most $n_{2}$ such that $\pi_{2}(\left\langle S_{2}\right\rangle )$
is Zariski dense in $G_{2}$. We have $\left|S_{1}\cup S_{2}\right|\leqslant n_{1}+n_{2}$.
We now separate to two cases. 
\begin{enumerate}
\item First, assume there is no isomorphism $f:G_{1}\to G_{2}$ such that,
for every $s=(x_{1},x_{2})\in S_{1}\cup S_{2}$, we have $f(x_{1})=x_{2}$.
We claim $\Gamma\coloneqq\left\langle S_{1}\cup S_{2}\right\rangle $
is Zariski dense in $G$. Set $R=\overline{\Gamma}^{Z}$, the Zariski
closure of $\Gamma$ in $G$. Since the image of an algebraic group
under a homomorphism is always closed, and $\pi_{i}(\Gamma)$ is
dense in $G_{i}$ for $i=1,2$ , we get that $\pi_{i}(R)=G_{i}$ for
$i=1,2$. Clearly, there is no isomorphism $f:G_{1}\to G_{2}$ such
that $R=\left\{ (x,f(x))\middle|x\in G_{1}\right\} $. Thus, by the
previous lemma, $R$ is not a proper subgroup, as needed.
\item Now, assume there is an isomorphism $f:G_{1}\to G_{2}$ such that,
for every $s=(x_{1},x_{2})\in S_{1}\cup S_{2}$, we have $f(x_{1})=x_{2}$.
Then, in fact, we don't need $S_{2}$ at all; we know $\left\langle \pi_{1}(S_{1})\right\rangle $
is Zariski dense in $G_{1}$, but we now also get that $\left\langle \pi_{2}(S_{1})\right\rangle =\left\langle f(\pi_{1}(S_{1}))\right\rangle $
is Zariski dense in $G_{2}$. Now, it is impossible that $f(x_{1})=x_{2}$
for every $(x_{1},x_{2})\in S$ (or else, we would get that $S$,
and hence $\left\langle S\right\rangle $, and hence $\overline{\left\langle S\right\rangle }^{Z}$,
are all inside the proper subgroup $\left\{ (x,f(x))\middle|x\in G_{1}\right\} $,
contrary to the assumption $\left\langle S\right\rangle $ is Zariski
dense in $G_{1}\times G_{2}$). Let $s=(x_{1},x_{2})\in S$ be some
element in $S$ such that $f(x_{1})\neq x_{2}$. Then $\left\langle S_{1}\cup\left\{ s\right\} \right\rangle $
is Zariski dense by the same argument as above, and $\left|S_{1}\cup\left\{ s\right\} \right|\leqslant n_{1}+1<n_{1}+n_{2}$.\qedhere
\end{enumerate}
\end{proof}
We can now finish the proof of the semisimple case:
\begin{thm}
\label{thm:ss-case}Let $\mathbf{G}$ be a connected semisimple algebraic
group and let $\mathbf{G}_{1}(\mathbb{C}),\dots,\mathbf{G}_{n}(\mathbb{C})$
be the simple quotients of $\mathbf{G}(\mathbb{C})$. Then 
\[
\zr(\mathbf{G})=\sum_{i=1}^{n}\zr(\mathbf{G}_{i})\leqslant10^{5}\cdot\sum_{i=1}^{n}\mathrm{rank}(\bG_{i})^{10}<\infty.
\]
\end{thm}

\begin{proof}
By \corref{isogeny}, $\zr(\mathbf{G})=\zr(\tilde{\mathbf{G}})$
for the algebraic universal covering $\tilde{\mathbf{G}}$ of $\mathbf{G}$,
so we may assume $\mathbf{G}$ is simply connected. In this case,
$\mathbf{G}(\mathbb{C})=\mathbf{G}_{1}(\mathbb{C})\times\cdots\times\mathbf{G}_{n}(\mathbb{C})$,
and each $\mathbf{G}_{i}$ is absolutely almost simple and simply-connected.
It now follows by induction on the previous lemma that $\zr(\mathbf{G})=\sum_{i=1}^{n}\zr(\mathbf{G}_{i})$.
Lastly, choosing a $\mathbb{Z}$-scheme structure of $\mathbf{G}_{i}$,
we get by Corollary \ref{cor:simply-connected-ab-al-simple-case}
that $\zr(\mathbf{G}_{i})\leqslant10^{5}\cdot\mathrm{rank}(\bG_{i})^{10}$
for every $i=1,\dots,n$.
\end{proof}

\subsection{Reductive Algebraic Groups\label{subsec:Reductive-Algebraic-Groups}}
\begin{lem}
\label{lem:rr-cn}We have $\zr((\mathbf{G}_{m})^{n})=n$.
\end{lem}

\begin{proof}
Fixing some $x\in\mathbb{C}^{\times}$ of infinite order and letting
$x_{i}\in(\mathbb{C}^{\times})^{n}$ be $x$ in the $i^{\text{th}}$
coordinate and $1$ everywhere else, it is easy to see $\left\{ x_{1},\dots,x_{n}\right\} $
is an irredundant Zariski generating set of ${\mathbf{G}_{m}}^{n}(\mathbb{C})=\left(\mathbb{C}^{\times}\right)^{n}$.

We now prove $\zr((\mathbf{G}_{m})^{n})\leqslant n$ by induction.
Assume $\zr((\mathbb{C}^{\times})^{m})\leqslant m$ for $m<n$ and
let $S\subseteq(\mathbb{C}^{\times})^{n}$ be a Zariski generating
set of size at least $n+1$. Then it admits an element $s\in S$ of
infinite order. Set $C=\overline{\left\langle s\right\rangle }^{Z}$;
then $(\mathbb{C}^{\times})^{n}/C\cong(\mathbb{C}^{\times})^{m}$
for some $m<n$, so that the projection of $S\backslash\left\{ s\right\} $
to the quotient group is redundant. This clearly means $S$ is redundant
as a Zariski generating set of $(\mathbb{C}^{\times})^{n}$.
\end{proof}
\begin{lem}
Let $G=\mathbf{G}(\mathbb{C})$ be a perfect algebraic group. Let
$\Gamma<(\mathbb{C}^{\times})^{n}\times G$ be an abstract subgroup.
Then $\Gamma$ is Zariski dense if and only if its projections to
$(\mathbb{C}^{\times})^{n}$ and to $G$ are Zariski dense.
\end{lem}

\begin{proof}
Clearly, the projections of a Zariski dense subgroup are Zariski dense.
Conversely, assume the projections of $\Gamma$ are dense. Denote
by $R$ the Zariski closure of $\Gamma$, and denote by $\pi_{1},\pi_{2}$
the projections. Then $\pi_{2}(R)=G$, and $[R,R]$ is contained in
$G$, so that $\pi_{2}([R,R])=[R,R]$. Thus, 
\[
R\supseteq[R,R]=\pi_{2}([R,R])=[\pi_{2}(R),\pi_{2}(R)]=[G,G]=G.
\]
This means that $\pi_{1}\res_{R}$ is onto and that $R$ contains
the kernel of $\pi_{1}$. In other words, $R=(\mathbb{C}^{\times})^{n}\times G$.
\end{proof}
\begin{thm}[Theorem \ref{thm:AmenableAndAlgebraic}]
\label{thm:reductive}Let $\mathbf{H}$ be a reductive algebraic
group, and let $\mathbf{G}_{1}(\mathbb{C}),\dots,\mathbf{G}_{n}(\mathbb{C})$
be the simple quotients of $\mathbf{H}(\mathbb{C})$. Then 
\[
\zr(\mathbf{H})=\dim\sr(\mathbf{H})+\sum_{i=1}^{n}\zr(\mathbf{G}_{i})\leqslant\dim\sr(\mathbf{H})+10^{5}\cdot\sum_{i=1}^{n}\rank(\bG_{i})^{10}<\infty.
\]
\end{thm}

\begin{proof}
Up to isogenies, we may assume 
\[
\mathbf{H}(\mathbb{C})=(\mathbb{C}^{\times})^{r}\times\mathbf{G}_{1}(\mathbb{C})\times\cdots\times\mathbf{G}_{n}(\mathbb{C}),
\]
where $r=\dim\sr(\mathbf{H})$. The inequality $\zr(\mathbf{H})\geqslant\dim\sr(\mathbf{H})+\sum_{i=1}^{n}\zr(\mathbf{G}_{i})$
is straightforward. The inequality $\zr(\mathbf{H})\leqslant\dim\sr(\mathbf{H})+\sum_{i=1}^{n}\zr(\mathbf{G}_{i})$
follows almost immediately from the previous lemma. Suppose 
\[
\gamma_{1},\dots,\gamma_{m}\in(\mathbb{C}^{\times})^{r}\times\mathbf{G}_{1}(\mathbb{C})\times\cdots\times\mathbf{G}_{n}(\mathbb{C})
\]
generate a Zariski dense subgroup of $(\mathbb{C}^{\times})^{r}\times\mathbf{G}_{1}(\mathbb{C})\times\cdots\times\mathbf{G}_{n}(\mathbb{C})$
with $m>r+\sum_{i=1}^{n}\zr(\mathbf{G}_{i})$. Denote by $\pi_{1},\pi_{2}$
the projections to $(\mathbb{C}^{\times})^{r}$ and $\prod_{i=1}^{n}\mathbf{G}_{i}(\mathbb{C})$
respectively. By \lemref{rr-cn}, there is a subset $S_{1}\subseteq\left\{ \gamma_{1},\dots,\gamma_{m}\right\} $
of size at most $r$ such that $\pi_{1}(\left\langle S_{1}\right\rangle )$
is Zariski dense in $(\mathbb{C}^{\times})^{r}$. There is also a
subset $S_{2}\subseteq\left\{ \gamma_{1},\dots,\gamma_{m}\right\} $
of size at most $\zr(\prod_{i=1}^{n}\mathbf{G}_{i})=\sum_{i=1}^{n}\zr(\mathbf{G}_{i})$
such that $\pi_{2}(\left\langle S_{2}\right\rangle )$ is Zariski
dense in $\prod_{i=1}^{n}\mathbf{G}_{i}(\mathbb{C})$. By the last
lemma, $\left\langle S_{1}\cup S_{2}\right\rangle $ is Zariski dense
in $(\mathbb{C}^{\times})^{r}\times\mathbf{G}_{1}(\mathbb{C})\times\cdots\times\mathbf{G}_{n}(\mathbb{C})$,
so we are done.
\end{proof}

\section{Amenable Lie Groups}

In this section, we are back to considering the\emph{ }redundancy
rank of topological groups. Recall that we say $g_{1},\dots,g_{n}\in G$
generate a topological group $G$ if the only closed subgroup containing
all of them is $G$ itself. For a Lie group $G$, we denote by $\sr(G)$
its solvable radical (the maximal connected normal solvable subgroup). 

Using the machinery developed by Abels and Noskov in \cite{AbeNos},
we are able to reduce the problem to simple Lie groups. We begin with
the compact simple case.
\begin{thm}
\label{thm:compact-case}Let $G$ be a connected compact Lie group
with simple quotients $G_{1},\dots,G_{n}$. Then 
\[
\trr(G)\leqslant\dim Z(G)+10^{5}\sum_{i=1}^{n}\mathrm{rank}(G_{i})^{10}<\infty.
\]
\end{thm}

\begin{proof}
Recall that $\trr(G)=\zrr^{\mathbb{R}}(\mathbf{G})\leqslant\zr(\mathbf{G})$
for the unique reductive algebraic group $\mathbf{G}$ such that $G=\mathbf{G}(\mathbb{R})$,
so the result follows from Theorem \ref{thm:reductive}.
\end{proof}
We first recall the results of Abels--Noskov, and then bound the
redundancy rank of connected amenable Lie groups. For the definition
and an easygoing introduction to amenability, see \cite{cohen2024invitation}.
\begin{defn}
If $G,H$ are topological groups, we say a surjective homomorphism
$f:G\to H$ is \emph{absolutely Gasch\"utz} when the following holds:
if $g_{1},\dots,g_{n}\in G$ are such that $f(g_{1}),\dots,f(g_{n})$
generate $H$, then $g_{1},\dots,g_{n}$ generate $G$.
\end{defn}

For example, if $G$ is a connected topological group, then it is
easy to see that a finite covering map $f:G\to H$ is always absolutely
Gasch\"utz. Therefore, when discussing the topological redundancy rank
of connected groups, one may ignore finite covering maps.

The following is immediate:
\begin{lem}
Let $f:G\to H$ be absolutely Gasch\"utz. Then $\trr(G)=\trr(H)$.
\end{lem}

An \emph{Abels--Noskov group} is a group of the form $(S\times A)\ltimes V$,
where $S$ is a connected semisimple Lie group, $A$ is a connected
abelian Lie group, $V$ is a finite-dimensional real vector space,
and the action of $S\times A$ on $V$ is given by a completely reducible
representation without nonzero fixed vectors. Abels and Noskov proved
the following:
\begin{thm}[Abels--Noskov]
\label{thm:AN-Machine}If $G$ is a connected Lie group, then there
is a normal subgroup $B\trianglelefteqslant G$ such that $f:G\to G/B$
is absolutely Gasch\"utz and such that $G/B$ is finitely covered by
an Abels--Noskov group $(S\times A)\ltimes V$. Moreover, $A$ is
isomorphic to $G/\overline{G'}$ and $S$ is locally isomorphic to
$G/\mathrm{R}_{s}(G)$. In particular, $\trr(S)=\trr(G/\mathrm{R}_{s}(G))$.
\end{thm}

\begin{proof}
The first sentence is \cites[Corollary 5.5 and Corollary 5.8]{AbeNos}.
The fact $A$ is isomorphic to $G/\overline{G'}$ is \cites[Lemma 7.2]{CohVig}.
The fact $S$ is locally isomorphic to $G/\mathrm{R}_{s}(G)$ follows
from the proof of \cites[Corollary 5.8 and Theorem 5.6]{AbeNos}.
The fact $\trr(S)=\trr(G/\mathrm{R}_{s}(G))$ follows from the fact
quotients by discrete central subgroups of connected semisimple Lie
groups are absolutely Gasch\"utz (\cites[Lemma 4.2]{CohVig}).
\end{proof}
\begin{rem}
For those who do not wish to get into the proof of \cites[Theorem 5.6]{AbeNos},
one can avoid referring to it as follows. Since $(S\times A)\ltimes V$
finitely covers $G/B$, it follows $S$ finitely covers a quotient
of $G/\mathrm{R}_{s}(G)$. Thus, modulo the (discrete) centres of
$S$ and $G/\mathrm{R}_{s}(G)$, $S$ is a direct factor of $G/\mathrm{R}_{s}(G)$.
It is therefore easy to see $\trr(S)\leqslant\trr(G/\mathrm{R}_{s}(G))$,
which is all that we will use below.
\end{rem}

\begin{lem}
\label{lem:AN-case}Let $(S\times A)\ltimes V$ be an Abels--Noskov
group. Then $\trr((S\times A)\ltimes V)\leqslant\trr(S)+\trr(A)+\dim V$.
\end{lem}

\begin{proof}
Set $L=S\times A$ and $n=\trr(L)$. Observe that $n=\trr(S)+\trr(A)$
by \cites[Lemma 5.1]{CohVig}. Now, let $(\ell_{1},v_{1}),\dots,(\ell_{m},v_{m})\in L\ltimes V$
be generators with $m>n+\dim V$. Up to rearrangement, we may assume
$\ell_{1},\dots,\ell_{n}$ generate $L$. By \cites[Lemma 6.4]{AbeNos},
$(\ell_{1},v_{1}),\dots,(\ell_{n},v_{n})$ generate a subgroup of
the form $L\ltimes W$ for an $L$-submodule $W\subseteq V$. In particular,
the subgroup they generate contains $L$, and therefore the subgroup
generated by $(\ell_{1},v_{1}),\dots,(\ell_{m},v_{m})$ is equal to
the subgroup generated by
\[
(\ell_{1},v_{1}),\dots,(\ell_{n},v_{n}),(1,v_{n+1}),\dots,(1,v_{m}).
\]
By assumption, $m-n>\dim V$. Therefore, up to rearrangement, $v_{m}$
is contained in the span of $v_{n+1},\dots,v_{m-1}$. It follows that
the subgroup generated by
\[
(\ell_{1},v_{1}),\dots,(\ell_{n},v_{n}),(1,v_{n+1}),\dots,(1,v_{m-1})
\]
 contains $(1,v_{m})$. Thus, this element is redundant.

The following in particular proves Theorem \ref{thm:master}(\ref{enu:amenable}).
\end{proof}
\begin{thm}
\label{thm:amenable-groups-bounds}If $G$ is a connected amenable
Lie group, then $\trr(G)$ is finite. More precisely, we have
\begin{align*}
\trr(G) & \leqslant\trr(G/\mathrm{R}_{s}(G))+\dim\sr(G)+\trr(G/\overline{G'})-\dim(G/\overline{G'})
\end{align*}
\end{thm}

\begin{proof}
By Theorem \ref{thm:AN-Machine}, there is a normal subgroup $B\trianglelefteqslant G$
such that $f:G\to G/B$ is absolutely Gasch\"utz and such that $G/B$
is finitely covered by an Abels--Noskov group $(S\times G/\overline{G'})\ltimes V$,
where $S$ is locally isomorphic to $G/\sr(G)$ and hence $\trr(S)=\trr(G/\sr(G))$.
Since $f$ and finite covering maps are absolutely Gasch\"utz, we have
(by Lemma \ref{lem:AN-case})
\[
\trr(G)=\trr\left((S\times G/\overline{G'})\ltimes V\right)\leqslant\trr(G/\sr(G))+\trr(G/\overline{G'})+\dim V.
\]
Plugging in 
\begin{align*}
\dim V & \leqslant\dim G-\dim G/\overline{G'}-\dim G/\sr(G)\\
 & =\dim\sr(G)-\dim G/\overline{G'},
\end{align*}
we get:
\[
\trr(G)\leqslant\trr(G/\sr(G))+\dim\sr(G)+\trr(G/\overline{G'})-\dim(G/\overline{G'}).\qedhere
\]
\end{proof}
\begin{rem}
By \cite[Theorem 1.5]{CohVig}, 
\[
\trr(G/\overline{G'})=2\dim G/\overline{G'}-\dim T
\]
for a maximal torus $T\leqslant G/\overline{G'}$. Thus, 
\[
\trr(G)\leqslant\trr(G/\mathrm{R}_{s}(G))+\dim\sr(G)+\dim(G/\overline{G}')-\dim T.
\]
Recall that $G/\mathrm{R}_{s}(G)$ is (up to finite coverings) a direct
product of connected compact simple Lie groups $K_{1},\dots,K_{n}$.
Theorem \ref{thm:compact-case} thus says that $\trr(G/\mathrm{R}_{s}(G))\leqslant10^{5}\sum_{i=1}^{n}\mathrm{rank}(K_{i})^{10}$.
Putting everything together, we get
\[
\trr(G)\leqslant10^{5}\sum_{i=1}^{n}\mathrm{rank}(K_{i})^{10}+\dim\sr(G)+\dim(G/\overline{G}')-\dim T.
\]
\end{rem}

\section{\label{subsec:Non-Amenable-Lie-Groups}Non-Amenable Lie Groups}

We have seen that $m(G)<\infty$ when $G$ is amenable. Towards proving
Theorem \ref{thm:master}(\ref{enu:non-amenable}), we show the following.
\begin{thm}
\label{thm:non-compact-simple m(G)=00003Dinf}Let $G$ be a connected
centre-free non-compact simple Lie group. Then $\mu(G)=\infty$. 
\end{thm}

Let us explain how Theorem \ref{thm:non-compact-simple m(G)=00003Dinf}
implies Theorem \ref{thm:master}(\ref{enu:non-amenable}):
\begin{proof}[Proof of Theorem \ref{thm:master}]
 Note that Theorem \ref{thm:master}(\ref{enu:amenable}) was already
established above (see Theorem \ref{thm:amenable-groups-bounds}).
For Theorem \ref{thm:master}(\ref{enu:non-amenable}), let $G$ be
a connected non-amenable Lie group. Then it admits a quotient $f:G\to G_{0}$
such that $G_{0}$ is a connected centre-free non-compact simple Lie
group, so that $\nrr(G_{0})=\infty$ by Theorem \ref{thm:non-compact-simple m(G)=00003Dinf}.
Therefore, there are arbitrarily large Nielsen irredundant generating
systems $g_{1},\dots,g_{n}\in G_{0}$. By the Gasch\"utz Lemma for connected
Lie groups \cite{CohVig}, we may find generating lifts $\bar{g}_{1},\dots,\bar{g}_{n}\in G$.
It is easy to see $\bar{g}_{1},\dots,\bar{g}_{n}$ are also Nielsen
irredundant (otherwise, we would get that $g_{1},\dots,g_{n}$ are
Nielsen redundant). Thus, $\nrr(G)=\infty$ as well.
\end{proof}

\subsection{Primitive-stable and Anosov representations}

The proof of Theorem \ref{thm:non-compact-simple m(G)=00003Dinf}
uses the notion of primitive-stable representations, introduced by
Minsky for rank one groups and extended to higher rank in \cite{minsky2013dynamics,kim2021primitive}.

Let $G$ be a connected semisimple linear Lie group without compact
factors. Fix a Cartan decomposition 
\[
G=K\exp(\overline{\mathfrak{a}^{+}})K,\qquad\mu_{G}:G\to\overline{\mathfrak{a}^{+}},
\]
and let $\Delta$ be the corresponding set of simple restricted roots.
Let $X=G/K$ be the associated symmetric space, and fix a base point
$o\in X$. The following definition should be compared with Theorem
1.3(4) and Remark 1.6(b) in \cite{GueritaudGuichardKasselWienhard2017AnosovProper}. 

\begin{defn}
\label{def:uru-root}Let $\Gamma$ be a hyperbolic group with a fixed
finite generating set $S$, and let $\mathcal{L}$ be a collection
of geodesic rays emanating from the identity in the Cayley graph of
$\Gamma$. For a non-empty subset of roots $\theta\subset\Delta$,
a representation $\rho:\Gamma\to G$ is said to be $(\mathcal{L},\theta)$\textit{-uniformly
regular} if there are constants $C,c>0$ such that, for every $(\gamma_{m})_{m\in\N}\in\mathcal{L}$,
\[
\alpha\bigl(\mu_{G}(\rho(\gamma_{m}))\bigr)\geq c\,|\gamma_{m}|_{S}-C,\qquad\forall\alpha\in\Sigma_{\theta}^{+},
\]
where $\Sigma_{\theta}^{+}$ is the collection of positive roots not
in $\mathrm{Span}(\Delta\backslash\theta)$.
\end{defn}

\begin{lem}
\label{lem:anosov passes to supgroups}Suppose that $H<G$ is a subgroup
locally isomorphic to $\mathrm{PSL}_{2}(\R)$ and denote the inclusion
map by $\iota$. Then there is a non-empty subset $\theta\subset\Delta$
such that, if $\rho:\Gamma\to H$ is $\mathcal{L}$-uniformly regular,
then $\iota\circ\rho$ is $(\mathcal{L},\theta)$-uniformly regular. 
\end{lem}

\begin{proof}
We may choose the Cartan decomposition of $G$ so that when restricting
to $\iota(H)$ we get a Cartan decomposition of $\iota(H)$. Then
in particular, the derivative $d\iota$ takes the positive Weyl cone
$\left\{ X_{t}\right\} _{t\geqslant0}$ of the Lie algebra of $H$
into the positive Weyl cone $\mathfrak{a}^{+}$ of $G$. Hence $d\iota(X_{t})=tv$
for some $0\ne v\in\mathfrak{a}^{+}$. Moreover, we have the following
compatibility between the Cartan projections 
\[
\mu_{G}\circ\iota=d\iota\circ\mu_{H}.
\]
Choose a root $\alpha\in\Delta$ for which $\alpha(v)\ne0$ and set
$\theta=\{\alpha\}$, so $\alpha'(v)\neq0$ for every $\alpha'\in\Sigma_{\theta}^{+}$.
Given $\left(\gamma_{m}\right)_{m\in\N}\in\mathcal{L}$, we have,
by assumption, that $\mu_{H}(\rho(\gamma_{m}))\in\R_{\geq0}$ grows
linearly in $m$ (independently of the geodesic in $\mathcal{L}$).
Therefore we may write 
\[
d\iota(\mu_{H}(\rho(\gamma_{m})))=t_{m}v
\]
for linearly growing $t_{m}\geq0$. Therefore
\[
\alpha'\left(\mu_{G}(\iota\circ\rho(\gamma_{m}))\right)=\alpha'\left(d\iota(\mu_{H}(\rho(\gamma_{m})))\right)=t_{m}\alpha'(v)\quad\forall\alpha'\in\Sigma_{\theta}^{+}.
\]
This shows $\iota\circ\rho$ is $(\mathcal{L},\theta)$-uniformly
regular. 
\end{proof}
We now restrict attention to two classes $\mathcal{L}$ of geodesic
rays.

Suppose first that $\mathcal{L}$ is taken to be the set of all geodesic
rays. Then the notion of a $(\mathcal{L},\theta)$-uniformly regular
representation coincides with that of $\theta$-\textit{Anosov representations}
\cite[Theorem 1.3]{GueritaudGuichardKasselWienhard2017AnosovProper}. 

Suppose now that $\Gamma=F_{n}$ is the free group on $n$ generators,
and $\mathcal{L}$ is taken to be all geodesic rays $\gamma\gamma\gamma\cdots$
for a primitive and cyclically-reduced $\gamma\in F_{n}$ (recall
that $\gamma$ is primitive if it is part of a generating set of $F_{n}$
of size $n$). Then the notion of a $(\mathcal{L},\theta)$-uniformly
regular representation coincides with that of $\theta$-\textit{primitive-stable
representation} as defined in \cite{minsky2013dynamics} in rank $1$
and in \cite{kim2021primitive} in general (see also \cite[Remark 1.6(b)]{GueritaudGuichardKasselWienhard2017AnosovProper}).
Note \cite{kim2021primitive} follows the terminology of \cite{KapovichLeebPorti2014MorseActions},
but the notions are equivalent due to the higher rank Morse Lemma
\cite[Theorem 1.3]{kapovich2018morse}.

The following lemma was mentioned in \cite{Kim20}, and it is proved
in \cite{KimPortiUniformlyPrimitiveStable}, a draft of which Inkang
Kim has kindly shared with us. For completeness, we provide a proof.
\begin{lem}
\label{lem:primitive-stable-free-factors}Let $\rho:F_{n}\to G$ be
$\theta$-primitive-stable. If $\Lambda<F_{n}$ is a proper free factor,
then $\rho|_{\Lambda}:\Lambda\to G$ is $\theta$-Anosov.
\end{lem}

\begin{proof}
Let $S=(s_{1},...,s_{n})$ be the generating set of $F_{n}$ fixed
a priori. Consider first the case where the free factor is $\Lambda_{k}=\langle s_{1},\ldots,s_{k}\rangle$
for some $1\leq k<n$. Given $a\in\Lambda_{k}$ and $\alpha\in\Sigma_{\theta}^{+}$,
we want to bound $\alpha\bigl(\mu_{G}(\rho(a))\bigr)$ from below
in terms of a linear function in $|a|_{S}$. We observe that the element
$as_{k+1}$ is primitive: it is obtained from $s_{k+1}$ by the sequence
of Nielsen transformations of left multiplication by the letters of
$a$. Moreover $as_{k+1}$ is cyclically reduced. As $\rho$ is $\theta$-primitive-stable
we have 
\[
\alpha\bigl(\mu_{G}(\rho(as_{k+1}))\bigr)\geq c\,|as_{k+1}|_{S}-C\qquad\forall\alpha\in\Sigma_{\theta}^{+}.
\]
for some constants $c,C>0$. Since $a$ and $as_{k+1}$ are neighbours,
$\left|\alpha\bigl(\mu_{G}(\rho(a))\bigr)-\alpha\bigl(\mu_{G}(\rho(as_{k+1}))\bigr)\right|$
is bounded by a constant that depends only on $\rho(s_{k+1})$. We
thus get that 
\[
\alpha\bigl(\mu_{G}(\rho(a))\bigr)\geq c\,|a|_{S}-C'\qquad\forall\alpha\in\Sigma_{\theta}^{+}
\]
for some new constant $C'$ that depends only on $\rho(s_{k+1})$.
This gives the desired statement for the free factor $\Lambda_{k}$.

Now, given an arbitrary $\sigma\in\mathrm{Aut}(F_{n})$, we may apply
the above to (the still primitive-stable) representation $\rho\circ\sigma$,
and get for all $a\in\Lambda_{k}$ that
\[
\alpha\bigl(\mu_{G}(\rho(\sigma(a)))\bigr)\geq c'\,|a|_{S}-C''\qquad\forall\alpha\in\Sigma_{\theta}^{+}.
\]
Observe that the constants $c',C''$ may depend on $\sigma$. Note
that $|\sigma(a)|_{S}\leq m_{\sigma}|a|_{S}$ where $m_{\sigma}=\max_{i}|\sigma(s_{i})|_{S}$,
so that we get
\[
\alpha\bigl(\mu_{G}(\rho(\sigma(a)))\bigr)\geq c'm_{\sigma}^{-1}\,|\sigma(a)|_{S}-C''\qquad\forall\alpha\in\Sigma_{\theta}^{+}.
\]
This shows that $\rho\mid_{\sigma(\Lambda_{k})}$ is $\theta$-Anosov
(with constants depending on $\sigma$). It is only left to note that
any proper free factor $\Lambda<F_{n}$ is of the form $\sigma(\Lambda_{k})$
for some $\sigma\in\mathrm{Aut}(F_{n})$ and $1\leq k<n$. 
\end{proof}

\subsection{Primitive-stable representations with dense image }

Fix $n$ and let $\Gamma=F_{n}$. Note that we have a natural identification
$\mathrm{Hom}(\Gamma,G)\cong G^{n}$ making it a locally compact second-countable
space endowed with the Haar measure class. We recall the following
facts regarding subsets of $\mathrm{Hom}(\Gamma,G)$:
\begin{enumerate}
\item The set $\mathcal{D}$ consisting of discrete and faithful representations
is closed, see \cite[Lemma 2.1]{minsky2013dynamics} and \cite{BreuillardGelander2003DenseFree}. 
\item The set $\mathcal{E}$ consisting of representations with dense image
is open \cites({Theorem 2.4(b)}){chirvasitu2021large}.
\item The union $\mathcal{D}\cup\mathcal{E}$ is co-null \cite[Lemma 2.1]{minsky2013dynamics}
and \cite{BreuillardGelander2003DenseFree}. 
\item The set $\mathcal{P}_{\theta}$ of all $\theta$-primitive-stable
representations is open \cite[Theorem 1.1]{kim2021primitive}.
\end{enumerate}
\begin{thm}
\label{prop:dense-primitive-stable}Let $G$ be a connected semisimple
linear Lie group without compact factors. There exists a nonempty
$\theta\subset\Delta$ such that, for every even $n\geq2$, there
is a $\theta$-primitive-stable representation $\rho:F_{n}\to G$
with dense image. 
\end{thm}

\begin{proof}
In \cite{minsky2013dynamics}, Minsky constructs for every even $n\geq2$
a primitive-stable representation $\rho_{0}:F_{n}\to\mathrm{PSL}_{2}(\R)$
with dense image; see \cite[Section 5]{minsky2013dynamics} and \cite[Section 9]{Lubotzky2011AutFnRepresentations}.
We may lift it to a primitive-stable representation $\rho_{1}:F_{n}\to\mathrm{SL}_{2}(\mathbb{R})$
with dense image. 

Since $G$ is non-compact and semisimple, its Lie algebra admits an
embedding of $\mathfrak{sl}_{2}(\mathbb{R})$, and hence we have a
nontrivial map $\widetilde{\mathrm{SL}}_{2}(\mathbb{R})\to G$, which
factors to a map $f:\mathrm{SL}_{2}(\R)\to G$ (since $G$ is linear).
Choose the $KA_{+}K$-decompositions so that $f(\mathrm{SO}_{2}(\mathbb{R}))\subseteq K$
and such that $f_{*}$ sends the positive Weyl chamber of $\mathrm{SL}_{2}(\mathbb{R})$
into the positive Weyl chamber of $G$. Let $\rho:=f\circ\rho_{1}$.

By Lemma \ref{lem:anosov passes to supgroups}, $\rho$ is $\theta$-primitive-stable
for some non-empty $\theta\subset\Delta$. The image of $\rho$ is
dense in the positive-dimensional subgroup $f(\mathrm{SL}_{2}(\R))$,
and in particular, it is not discrete. 

By openness of primitive stability, $\rho$ has an open neighbourhood
$U$ consisting of $\theta$-primitive-stable representations. Since
$\mathcal{D}$ is closed and $\rho\notin\mathcal{D}$, we may shrink
$U$ so that $U\cap\mathcal{D}=\varnothing$. But as $\mathcal{D}\cup\mathcal{E}$
is conull, it must be that $U\cap\mathcal{E}\neq\varnothing$. Choosing
$\rho'$ in this intersection gives a $\theta$-primitive-stable representation
with dense image.
\end{proof}
\begin{proof}[Proof of Theorem \ref{thm:non-compact-simple m(G)=00003Dinf}]
Fix an even $n\geq2$. By Theorem \ref{prop:dense-primitive-stable},
we can choose a $\theta$-primitive-stable representation $\rho:F_{n}\to G$
with dense image. Let $S=(s_{1},\ldots,s_{n})$ be the fixed free
basis and set 
\[
X_{\rho}=(\rho(s_{1}),\ldots,\rho(s_{n}))\in G^{n}.
\]
This tuple topologically generates $G$. We claim that it is Nielsen
irredundant. Suppose not. Then there are $\sigma\in\mathrm{Aut}(F_{n})$
and an index $i$ such that $(\rho(\sigma(s_{1})),\ldots,\rho(\sigma(s_{n})))$
topologically generates $G$ even after deleting the $i$-th entry.
Equivalently, for the proper free factor $\Lambda=\sigma(\langle s_{1},\ldots,\widehat{s_{i}},\ldots,s_{n}\rangle)$,
the subgroup $\rho(\Lambda)$ is dense in $G$. On the other hand,
Lemma \ref{lem:primitive-stable-free-factors} implies that $\rho|_{\Lambda}$
is $\theta$-Anosov. Hence $\rho(\Lambda)$ is discrete, and in particular
not dense, contradiction. Thus $X_{\rho}$ is Nielsen irredundant.
We have proved $\mu(G)\geq n$ for every even $n$, and hence $\mu(G)=\infty$.
Since $\mu(G)\leq m(G)$, also $m(G)=\infty$.
\end{proof}
\begin{proof}[Proof of Theorem \ref{Thm:domain of discontinuity}]
Consider the set of $\theta$-primitive-stable representations $\mathcal{P}_{\theta}$
and the set of dense representations $\mathcal{E}$ in $\mathrm{Hom}(F_{n},G)$.
Their images in the character variety $\mathcal{X}(F_{n},G)=\mathrm{Hom}(F_{n},G)//G$
are open. By Theorem \ref{prop:dense-primitive-stable}, their intersection
is nonempty. Finally $\mathrm{Out}(F_{n})$ acts properly discontinuously
on the image of $\mathcal{P}_{\theta}$ in $\mathcal{X}(F_{n},G)$
\cite[Theorem 1.1]{kim2021primitive}.
\end{proof}

\section{Examples\label{sec:Examples}}
\begin{proof}[Proof of Theorem \ref{thm:rotations}]
We argue case by case:
\begin{itemize}
\item The compact group $\mathrm{SO}(3)$ is the real form of $\mathrm{PGL}_{2}$
which is in turn isogenous to $\mathrm{SL}_{2}$. Moreover, by \cite{gilman1977finite}
(see also \cite{busch2025wiegold}), we have $\mu(\mathrm{PSL}_{2}(\mathbb{F}_{p}))=2$
for all primes $p\geq5$. It is not hard to see that the center of
a perfect group is a Frattini subgroup, so we also get $\mu(\mathrm{SL}_{2}(\mathbb{F}_{p}))=2$
for all primes $\geq5$. As a result
\begin{align*}
2\leq\mu(\mathrm{SO}(3)) & \leq\mu^{\C}\left(\mathrm{PGL}_{2}\right)=\mu^{\C}(\mathrm{SL}_{2})\\
 & \leq\limsup_{p\in\N,\text{ prime}}\mu(\mathrm{SL}_{2}(\mathbb{F}_{p}))\\
 & =\limsup_{p\in\N,\text{ prime}}\mu(\mathrm{PSL}_{2}(\mathbb{F}_{p}))=2.
\end{align*}
\item We have $\rr(\mathrm{SL}_{2}(\mathbb{F}_{p}))=3$ for almost every
prime $p$ by \cite{jambor2013minimal}, which as above implies $\rr(\mathrm{SO}(3))\leq3$.
On the other hand, it is not hard to see that $\mathrm{SO}(3)$ can
be generated by three $\pi$-angled rotations. Such a generating set
must be irredundant, because a group generated by a pair of order
$2$ elements must be virtually abelian (because the free product
$\Z/2\Z*\Z/2\Z$ is isomorphic to the infinite dihedral group $\Z/2\Z\ltimes\Z$).
Hence 
\[
m(\mathrm{SO}(3))=3.
\]
\item Since $\mathrm{SU}(2)$ is isogenous to $\mathrm{SO}(3)$, and $\mathrm{U}(2)$
is isogenous to $\mathrm{SU}(2)\times\R/\Z$, we get 
\[
\rr(\mathrm{U}(2))=m(\mathrm{SO}(3))+m(\R/\Z)=4.
\]
\item Since $\mathrm{SO}(4)$ is isogenous to $\mathrm{SO}(3)\times\mathrm{SO}(3)$
we get
\[
\rr(\mathrm{SO}(4))=2\rr(\mathrm{SO}(3))=6.
\]
\item By \cite{keen2012independent}, $\rr(\mathrm{SL}_{3}(\mathbb{F}_{p}))\leq6$
for all primes $p$. Since $\mathrm{SU}(3)$ is the real compact form
of $\mathrm{SL}_{3}$, we deduce in a similar fashion to the computation
for $\mathrm{SO}(3)$ that
\[
\rr\left(\mathrm{SU}(3)\right)\leq6.\qedhere
\]
\end{itemize}
\end{proof}
\printbibliography

\vspace{0.5cm}

\noindent{\textsc{Department of Mathematics, University of California San Diego, 9500 Gilman Drive, La Jolla, CA
92093, USA}}
\vspace{0.5cm}

\noindent{\textit{Email address:} \texttt{ivigdorovich@ucsd.edu}}

\noindent{\textit{Webpage:} \texttt{https://sites.google.com/view/itamarv}} \\
\vspace{0.5cm}

\noindent{\textsc{Department of Mathematics, Weizmann Institute of Science, 234 Herzl Street, Rehovot 7610001, Israel}}
\vspace{0.5cm}

\noindent{\textit{Email address:} \texttt{tal.cohen@weizmann.ac.il}}

\end{document}